\documentclass[preprint,review,10pt]{amsart}
\usepackage{amsmath,amssymb,amsfonts,xspace}
\newtheorem{theorem}{Theorem}[section]

\theoremstyle{definition}
\newtheorem{definition}[theorem]{Definition}
\newtheorem{example}[theorem]{Example}

\newtheorem{proposition}[theorem]{Proposition}
\newtheorem{corollary}[theorem]{Corollary}

\theoremstyle{remark}
\newtheorem{remark}[theorem]{Remark}

\numberwithin{equation}{section}

\begin{document}

\title{ merging the $N$-hyperideals and $J$-hyperideals in one frame
 }

\author{M. Anbarloei}
\address{Department of Mathematics, Faculty of Sciences,
Imam Khomeini International University, Qazvin, Iran.
}

\email{m.anbarloei@sci.ikiu.ac.ir }





\begin{abstract}
The notions of $N$-hyperideals and $J$-hyperideals as two new classes of hyperideals were recently defined in the context of  Krasner $(m,n)$-hyperrings.  These concepts are created on the basis of the intersection of all $n$-ary prime hyperideals and the intersection of all maximal  hyperideals, respectively.  Despite being vastly different in many aspects, they share numerous similar properties.
The aim of this research
work is to merge them under one frame called $n$-ary $\delta(0)$-hyperideals where the function $\delta$ assigns to each hyperideal of a Krasner $(m,n)$-hyperring  a hyperideal of the same hyperring. We give various properties of $n$-ary $\delta(0)$-hyperideals  and use them to characterize certain classes of hyperrings such as hyperintegral domains and local hyperrings. Moreover, we introduce the notions of $(s,n)$-absorbing $\delta(0)$-hypereideals  and weakly $(s,n)$-absorbing $\delta(0)$-hypereideals.

\end{abstract}
\maketitle
\section{Introduction}
Although two of the most important structures
in commutative algebra, namely prime and primary ideals, differ significantly in various aspects, they share   several similar properties. The question was whether a unified approach to studying these two structures is possible. In \cite{bmb2}, 
 Dongsheng presented the concept of $\delta$-primary ideals  unifing  the prime and primary ideals under one frame in a commutative ring. Furthermore, Fahid and Dongsheng put two concepts of 2-absorbing ideals and 2-absorbing primary ideals in one frame called 2-absorbing $\delta$-primary ideals in \cite{bmb3}.

Hyperstructures are algebraic structures that have at least one multi-valued operation, known as a hyperoperation. The idea of the algebraic hyperstructures  goes back to
Marty$^,$s research work on hypergroups presented during the $8^{th}$ Congress of the Scandinavian Mathematicians in 1934. Since then,
numerous papers and books concerning this topic have
been written. They can be seen in  \cite {s2, s3, davvaz1, davvaz2, s4, jian}. 
$n$-ary hypergroups as an extension of the notion of a hypergroup in the sense of Marty were defined in \cite{s9}. For more study on
$n$-ary structures refer to \cite{l1, l2, l3}.
The notion of $(m,n)$-hyperrings was presented by Mirvakili and Davvaz  in \cite{cons}. Krasner $(m,n)$-hyperrings as a subclass of $(m,n)$-hyperrings were given in \cite{d1}. Some review of
the hyperrings can be found in \cite{asadi, ma, d1, nour,  rev1}.
Ameri and Norouzi defined $n$-ary prime and $n$-ary primary hyperideals in Krasner $(m, n)$-hyperring in \cite{sorc1} and Hila et al. generalized these concepts and studied $(k,n)$-absorbing and $(k,n)$-absorbing primary hyperideals in  \cite{rev2}. Also, Davvaz et al. defined weakly   $(k,n)$-absorbing and weakly $(k,n)$-absorbing primary hyperideals in \cite{www}. As a recent study, the concept of  $N$-hyperideals was introduced in \cite{mah6} and  some properties of them  have been  investigated analogous with prime hyperideals. Let $G$ be  a commutative Krasner $(m,n)$-hyperring. A hyperideal $A$ of  $G$ is called an $n$-ary $N$-hyperideal if  $g(x_1^n)$ is in $A$ for $x_1^n \in G$ such that $x_i$ is not in the intersection of the all $n$-ary prime hyperideals of $G$ for some $i \in \{1,\cdots,n\}$, then $g(x_1^{i-1},1_G,x_{i+1}^n)$ is in $A$. The notion of $J$-hyperideals as an extension of $N$-hyperideals was defined in \cite{mah4}. A hyperideal $A$ of  $G$ refers to an $n$-ary $J$-hyperideal if  $g(x_1^n) \in A$ for $x_1^n \in G$ such that $x_i$ is not in the intersection of the all maximal hyperideals of $G$ for some $i \in \{1,\cdots,n\}$, then $g(x_1^{i-1},1_G,x_{i+1}^n) \in A$. The two structures have been treated differently  and their properties were independently proven in the mentioned references. Despite their differences in many aspects, they share several similar properties. It is therefore essential to explore the possibility of a unified approach to studying these two structures.

In this paper, we aim to introduce the notion of $\delta(0)$-hyperideals unifying the $N$-hyperideals and $J$-hyperideals under one frame. This method effectively demonstrates the similarities and relationships between the two structures. The paper is orgnized as follows. In Section 2, we recall  the necessary background, and we fix notation. In Section 3, we define $n$-ary $\delta(0)$-hyperideals as a frame containing the concepts of $N$-hyperideals and $J$-hyperideals. However, Example \ref{mesal} shows that an $n$-ary $\delta(0)$-hyperideal may not be an $n$-ary  $N$-hypeideal. We give equivalent characterizations for an $n$-ary $\delta(0)$-hyperideals in Theorem \ref{1/7}. It is shown that an $n$-ary $\delta(0)$-hyperideal of a Krasner $(m,n)$-hyperring is contained in $\delta(0)$ in  Theorem \ref{1}.  We conclude that if $G$ admits an $n$-ary $\delta(0)$-hyperideal, then there exists an $n$-ary $\delta(0)$-hyperideal of $G$ such that there is no $\delta(0)$-hyperideal containing it in Theorem \ref{admits}. Moreover, we
examine  $n$-ary $\delta(0)$-hyperideals under various contexts of constructions such as cartesian products, homomorphic images,  localizations. The penultimate section is dedicated to the notion of $(s,n)$-absorbing $\delta(0)$-hyperideals. Some properties concerning it are investigated. The final section deals with issues surrounding the notion of weakly $(s,n)$-absorbing $\delta(0)$-hyperideals as an extension of the $(s,n)$-absorbing $\delta(0)$-hyperideals. We finish this section with  a copy of Nakayama,s lemma for a kind of weakly $(s,n)$-absorbing
$\delta(0)$-hypereideals.
In this paper, all the hyperrings used are commutative Krasner $(m,n)$-hyperrings with scalar identity.
\section{Preliminaries}
In this section, we give deﬁnitions and notations that will be used in this paper.

An $n$-ary hyperoperation $``f"$ on nonempty set $G$ is a mapping of $G^n$ into the family of all non-empty subsets of $G$. If $``f"$ is an $n$-ary hyperoperation on $G$, then $(G, f)$ is called an $n$-ary hypergroupoid. We can extend the $n$-ary hyperoperation on $G$ to non-empty subsets of $G$ as follows. For $G_1,\cdots, G_n \subseteq G$, then
\[f(G^n_1) = f(G_1,..., G_n) = \bigcup \{f(a_1,\cdots,a_n) \ \vert \ a_i \in G_i, i = 1,..., n \}.\]

Let us use the notation $a^j_i$ instead of   the sequence $a_i, a_{i+1},..., a_j$. Then we have $f(a_1,\cdots, a_i, b_{i+1},\cdots, b_j, c_{j+1},\cdots, c_n)= f(a^i_1, b^j_{i+1},c^n_{j+1})$ and  $f(a_1,\cdots, a_i, \underbrace{b,\cdots,b}_{j-i},\break  c_{j+1},\cdots, c_n)=f(a^i_1, b^{(j-i)}, c^n_{j+1})$ The above notation is the  empty symbol where $j< i$. Let $f$ be an $n$-ary hyperoperation.  Then $r$-ary hyperoperation $f_{(l)}$ for $r = l(n- 1) + 1$ is given by $f_{(l)}(a_1^r) = \underbrace{f(f(..., f(f}_l(a^n _1), a_{n+1}^{2n -1}),...), a_{r-n+1}^{r})$. An $n$-ary hypergroupoid $(G, f)$ is commutative if $f(a_1^n) = f(a_{\sigma(1)}^{\sigma(n)})$ for all $a_1^n \in G $ and $ \sigma \in \mathbb{S}_n$.
An $n$-ary semihypergroup is an $n$-ary hypergroupoid $(G, f)$, which is associative, that is $f(a^{i-1}_1, f(a_i^{n+i-1}), a^{2n-1}_{n+i}) = f(a^{j-1}_1, f(a_j^{n+j-1}), a_{n+j}^{2n-1})$  for  $1 \leq i < j \leq n$ and  $a_1^{2n-1} \in G$. If the equation $a \in f(a_1^{i-1}, x_i, a_{ i+1}^n)$ in an $n$-ary hypergroupoid $(G, f)$ has a solution $x_i \in G$
for all  $a_1^{i-1}, a_{ i+1}^n,a \in G$ and $i \in \{1,\cdots,n\}$, then $(G,f)$ is called an $n$-ary quasihypergroup. An  $n$-ary hypergroup is an $n$-ary semihypergroup that is an $n$-ary quasihypergroup. A non-empty subset $H$ of an $n$-ary hypergroup $G$ is called
an $n$-ary subhypergroup of $G$ if $(H,f)$ is an $n$-ary hypergroup \cite{d1}.

\begin{definition}
\cite{d1} An algebraic hyperstructure $(G, f, g)$, or simply $G$, is said to be a Krasner $(m, n)$-hyperring  if it satisfies the following axioms:
\begin{itemize} 
\item[\rm{(1)}]~ $(G, f$) is a canonical $m$-ary hypergroup, that is 
\begin{itemize} 
\item[\rm{(i)}]~there exists a unique $e \in G$, such that for each $a \in G, f(a, e^{(m-1)}) = \{a\}$;
\item[\rm{(ii)}]~for all $a \in G$ there exists a unique $a^{-1} \in G$ with  $e \in f(a, a^{-1}, e^{(m-2)})$;
\item[\rm{(iii)}]~if $a \in f(a^m _1)$, then  $a_i \in f(a, a^{-1},..., a^{-1}_{ i-1}, x^{-1}_ {i+1},..., a^{-1}_ m)$ for all $i\in \{1,\cdots,m\}$.
\end{itemize} 
\item[\rm{(2)}]~ $(G, g)$ is a $n$-ary semigroup;
\item[\rm{(3)}]~
$g(a^{i-1}_1, f(b^m _1 ), a^n _{i+1}) = f(g(a^{i-1}_1, b_1, a^n_{ i+1}),..., g(a^{i-1}_1, b_m, a^n_{ i+1}))$ for every $a^{i-1}_1 , a^n_{ i+1}, b^m_ 1 \in G$, and $i \in \{1,\cdots,n\}$;
\item[\rm{(4)}]~  $g(0, a^n _2) = g(a_2, 0, a^n _3) = ... = g(a^n_ 2, 0) = 0$ for all $a^n_ 2 \in G$.
\end{itemize} 
\end{definition}
We assume throughout this paper that $G$ is a commutative Krasner $(m,n)$-hyperring with scalar identity $1_G$, that is $g(a,1_G^{(n-2)})=a$ for all $a \in G$. 

If $(R, f, g)$ is a Krasner $(m, n)$-hyperring such that $\varnothing \neq  R \subseteq G$, then  $R$ is called a subhyperring of $G$. If $\varnothing \neq A \subseteq G$ such that $(A, f)$ is an $m$-ary subhypergroup
of $(G, f)$ and $g(a^{i-1}_1, A, a_{i+1}^n) \subseteq I$  for  $a^n _1 \in R$, $i \in \{1,\cdots,n\}$, then $A$ is said to be  a hyperideal of $G$.
Let $A$ be a hyperideal of $G$. Then the set $G/A=\{f(a_1^{i-1},A,a_{i+1}^n) \ \vert \ a_1^{i-1},a_{i+1}^n \in G\}$ is a Krasner $(m, n)$-hyperring with $m$-ary hyperoperation and $n$-hyperoperation $f$ and $g$, respectively \cite{sorc1}.

\begin{definition} \cite{sorc1} Let $a \in G$. Then  $\langle a \rangle$ denotes the hyperideal generated by $a$  and defined by $\langle a \rangle=g(G,a,1^{(n-2)})=\{g(r,a,1_G^{(n-2)}) \ \vert \ r \in G\}$.
\end{definition}
\begin{definition} \cite{sorc1} A hyperideal $A$ of $G$ refers to a  maximal hyperideal if  $A \subseteq B \subseteq G$ for every hyperideal $B$ of $G$ implies that $B=A$ or $B=G$.
\end{definition}
$Max(G)$ denotes the set of all maximal hyperideal of $G$. Jacobson radical of  $G$, being denoted by $J(G)$,  is the intersection of all maximal hyperideals of $G$. If $G$ does not have any maximal hyperideal, then we define $J(G)=G$. Moreover, if $G$ has just one maximal hyperideal, then $G$ is called local.
\begin{definition} \cite{sorc1} 
Let  $a \in G$. The element is  invertible if there exists $b \in G$ with $1_G=g(a,b,1_G^{(n-2)})$. 
\end{definition}






\begin{definition} \cite{d1}
Let $(G_1, f_1, g_1)$ and $(G_2, f_2, g_2)$ be two Krasner $(m, n)$-hyperrings. A mapping
$\psi : G_1 \longrightarrow G_2$ is called a homomorphism if for all $x^m _1 \in G_1$ and $y^n_ 1 \in G_1$ we have
\begin{itemize} 
\item[\rm{(i)}]~$\psi(f_1(x_1,..., x_m)) = f_2(\psi(x_1),...,\psi(x_m)),$
\item[\rm{(ii)}]~$\psi(g_1(y_1,..., y_n)) = g_2(\psi(y_1),...,\psi(y_n)), $
\item[\rm{(iii)}]~$\psi(1_{G_1})=1_{G_2}.$
\end{itemize}
\end{definition}







\begin{definition} \cite{16} 
Assume that $M \neq \varnothing$. Then $(M, h, k)$ is an $(m, n)$-hypermodule over $G$,  if 
\begin{itemize} 
\item[\rm{(i)}]~ $(M, h)$ is a canonical $m$-ary hypergroup. 
 \item[\rm{(ii)}]~ The map 

$\hspace{3cm}k:\underbrace{G \times ... \times G}_{n-1} \times M\longrightarrow 
P^*(M)$\\
statisfied the following conditions:
\end{itemize}
\begin{itemize} 
\item[\rm{(1)}]~ 
$k(a_1^{n-1},h(x_1^m))=h(k(a_1^{n-1}
,x_1),...,k(a_1^{n-1}
,x_m))$
\item[\rm{(2)}]~ $k(a_1^{i-1},f(b_1^m),a_{i+1}^{n-1},x)=h(k(a_1^{i-1}
,b_1,r_{i+1}^{n-1},x),...,k(r_1^{i-1}
b_m,r_{i+1}^{n-1},x))$
\item[\rm{(3)}]~  $k(a_1^{i-1},g(a_i^{i+n-1}),a_{i+m}^{n+m-2},x)=
k(a_1^{n-1},k(a_m^{n+m-2},x))$
\item[\rm{(4)}]~$ 0=k(r_1^{i-1},0,r_{i+1}^{n-1},x)$.
\end{itemize} 
\end{definition} 
\section{$n$-ary $\delta(0)$-hyperideals}
A hyperideal expansion $\delta$ of $G$ is a fuction which assigns to each hyperideal $A$ of $G$ a hyperideal $\delta(A)$ of $G$ such that $A \subseteq \delta(A)$, and $B \subseteq C$ for  hyperideals $B, C$ of $G$ implies $\delta(B) \subseteq \delta(C)$. For example, the functions $\delta_0, \delta_1, \delta_G$ and $\delta_M$  from the set of all hyperideals of $G$ to the same set, defined by $\delta_0(A)=A, \delta_1(A)=rad(A),\delta_G(A)=G$ and $\delta_M(A)=\cap_{A \subseteq B \in Max(G)}B$, are hyperideal expansions of $G$. In \cite{mah3}, the $n$-ary prime and $n$-ary primary hyperideals were put in a frame called $n$-ary $\delta$-primary hyperideals. In this section, we introduce and study the notion of $n$-ary $\delta(0)$-hyperideals  in a commutative Krasner $(m,n)$-hyperring $G$ which unify the $n$-ary $N$-hyperideals and $n$-ary $J$-hyperideals. 
\begin{definition} 
Let $A$ be a proper hyperideal of $G$ and $\delta$ a hyperideal expansion of $G$.  $A$ refers to an $n$-ary $\delta(0)$-hyperideal if whenever $x_1^n \in G$ with $g(x_1^n) \in A$  and  $g(x_1^{i-1},1_G,x_{i+1}^n) \notin  \delta(0)$ for some $i \in \{1,\cdots,n\}$, then $x_i \in A$.
\end{definition}
\begin{example} 
 Consider the  hyperring $([\alpha,\infty) \cup \{0\},+,\cdot)$ where $\alpha \geq 1$, $``\cdot"$ is the usual multiplication and $``+"$ is defined by
\[\hspace{3cm}
a + b=
\begin{cases}
[a,\infty) \cup \{0\} & \text{if $a=b\neq 0$.}\\
b + a=\{a\} & \text{if $b=0$},\\
\{\min\{a,b\}\} & \text{if $a \neq 0, b \neq 0$ and $a \neq b$.}\\
\end{cases}\]
Note that $([\alpha,\infty) \cup \{0\},+,\cdot)$ is a Krasner $(m,n)$-hyperring with $f(a_1^m)=\Sigma_{i=1}^ma_i$ and $g(b_1^n)=\Pi_{i=1}^n b_i$ for all $a_1^m,b_1^n \in [\alpha,\infty) \cup \{0\}$.
The hyperideal $0$ is an $n$-ary $\delta_1(0)$-hyperideal of $[\alpha,\infty) \cup \{0\}$.
\end{example} 
The next example shows that $n$-ary $\delta(0)$-hyperideals and $n$-ary $N$-hyperideals are different notions.
\begin{example} \label{mesal}
Note that $(\mathbb{Z}/6\mathbb{Z},f,g)$ is a commutative Krasner $(m,n)$-hyperring in which $f$ and $g$ are usual
addition and multiplication, respectively. Consider $\delta(I)=\{x \in \mathbb{Z}/6\mathbb{Z} \ \vert \ g(x,3\mathbb{Z}/6\mathbb{Z},1_{\mathbb{Z}/6\mathbb{Z}}^{(n-2)}) \subseteq I\}$. Let $A=2\mathbb{Z}/6\mathbb{Z}$ for every hyperideal $I$ of $G$. Then $A$ is an $n$-ary $\delta(0_{\mathbb{Z}/6\mathbb{Z}})$-hyperideal. However, $A$ is not an $n$-ary $N$-hyperideal.
\end{example}
Our first theorem gives a characterization of $\delta(0)$-hyperideals.
\begin{theorem} \label{1/7}
Let $A$ be a proper hyperideal of $G$. Then we have the following equivalent statements:
\begin{itemize} 
\item[\rm{(i)}]~ $A$ is an $n$-ary $\delta(0)$-hyperideal of $G$;
\item[\rm{(ii)}]~$A=E_x$ where $E_x=\{y \in G \ \vert \ g(x,y,1^{(n-2)}) \in A\}$ for each $x \notin \delta(0)$;
\item[\rm{(iii)}]~$g(A_1^n) \subseteq A$ for hyperideals $A_1^n$ of $G$ such that  $g(A_1^{i-1},1_G,A_{i+1}^n) \bigcap (G-\delta(0)) \neq \varnothing$ for some $i \in \{1,\cdots,n\}$ implies that $A_i \subseteq A$
\end{itemize} 
\end{theorem}
\begin{proof}
(i) $\Longrightarrow$ (ii) Let $A$ be  an $n$-ary $\delta(0)$-hyperideal of $G$. We always have $A \subseteq E_x$ for each $x \in G$. Suppose that $y \in E_x$ for  $x \notin  \delta(0)$. This means $g(x,y,1^{(n-2)}) \in A$. Since $A$ is an $n$-ary $\delta(0)$-hyperideal of $G$ and $x=g(x,1^{(n-1)}) \notin \delta(0)$, we get the result that $y \in A$ and so $A=E_x$.

(ii) $\Longrightarrow$ (iii) Let $g(A_1^n) \subseteq A$ for hyperideals $A_1^n$ of $G$ and  $g(A_1^{i-1},1_G,A_{i+1}^n) \bigcap (G-\delta(0)) \neq \varnothing$ for some $i \in \{1,\cdots,n\}$. Since $g(A_1^{i-1},1_G,A_{i+1}^n) \bigcap (G-\delta(0)) \neq \varnothing$ for some $i \in \{1,\cdots,n\}$,  there exists $x_j \in A_j$ for all $j \in \{1,\cdots,n\}$ and $j \neq i$ such that $g(x_1^{i-1},1_G,1_{i+1`}^n) \notin \delta(0)$. Since $g(g(x_1^{i-1},1_G,x_{i+1}^n),A_i,1^{(n-2)})=g(x_1^{i-1},A_i,x_{i+1}^n)\subseteq A$, we have $A_i \subseteq E_{g(x_1^{i-1},1_G,x_{i+1}^n)}$. Now, by (ii) we get $A_i \subseteq A$.

(iii) $\Longrightarrow$ (i) Let $g(x_1^n) \in A$ for $x_1^n \in G$ such that $g(x_1^{i-1},1_G,x_{i+1}^n) \notin \delta(0)$. Put $A_i=\langle x_i \rangle$ for each $i \in \{1,\cdots,n\}$.  Since $g(x_1^n) \in g(A_1^n) \subseteq A$ and $g(x_1^{i-1},1_G,x_{i+1}^n) \in g(A_1^{i-1},1_G,A_{i+1}^n) \bigcap (G-\delta(0))$, we get the result that $A_i \subseteq A$. This means that $x_i \in A$. Hence $A$ is an $n$-ary $\delta(0)$-hyperideals, as desired.
\end{proof}
\begin{theorem}
Assume that $A_1^{n-1}, A$ and $B$  are some hyperideals of $G$ such that $g(A_1^{n-1},1_G) \bigcap (G-\delta(0)) \neq \varnothing$. Then:
\begin{itemize} 
\item[\rm{(i)}]~ If $A$ and $B$ are $n$-ary $\delta(0)$-hyperideals of $G$ that $g(A_1^{n-1},A)=g(A_1^{n-1},B)$, then $A=B$.
\item[\rm{(ii)}]~ If $g(A_1^{n-1},A)$ is an $n$-ary $\delta(0)$-hyperideal of $G$, then $g(A_1^{n-1},A)=A$.
\end{itemize} 
\end{theorem}
\begin{proof}
(i) Suppose that  $g(A_1^{n-1},A)=g(A_1^{n-1},B)$ for hyperideals $A_1^{n-1},A$ and $B$ of $G$. Then we have $g(A_1^{n-1},B) \subseteq A$. Since $A$ is an  $n$-ary $\delta(0)$-hyperideal of $G$ and $g(A_1^{n-1},1_G) \bigcap (G-\delta(0)) \neq \varnothing$, we get the result that $B \subseteq A$ by Theorem \ref{1/7}. Similarly, we obtain $A \subseteq B$ and so $A=B$.

(ii) Let $g(A_1^{n-1},A)$ be an $n$-ary $\delta(0)$-hyperideal of $G$ for hyperideals $A_1^{n-1},A$ of $G$. Since $g(A_1^{n-1},A) \subseteq g(A_1^{n-1},A)$ and $g(A_1^{n-1},1_G) \bigcap (G-\delta(0)) \neq \varnothing$, we have $A \subseteq g(A_1^{n-1},A)$. Since $g(A_1^{n-1},A) \subseteq A$, we get $g(A_1^{n-1},A)=A$.
\end{proof}
\begin{proposition}
Let $\delta$ and $\gamma$ be two hyperideal expansions of $G$ such that $\gamma(0) \subseteq \delta(0)$.  Then every  $n$-ary $\gamma(0)$-hyperideal of $G$ is an $n$-ary $\delta(0)$-hyperideal.
\end{proposition}
\begin{theorem} \label{1/1}
Assume that $R$ is a non-empty subset of $G$. If $A$ is an $n$-ary $\delta(0)$-hyperideal of $G$ such that $R \nsubseteq A$, then $E_R=\{x \in G \ \vert \ g(x,R,1_G^{(n-2)}) \subseteq A \}$ is an $n$-ary $\delta(0)$-hyperideal of $G$.
\end{theorem}
\begin{proof}
It is clear that $E_R \neq G$. Suppose that $g(x_1^n) \in E_R$ for $x_1^n \in G$ such that $g(x_1^{i-1},1_G,x_{i+1}^n) \notin \delta(0)$ for some $i \in \{1,\cdots,n\}$. Then we conclude that  $g(g(x_i,r,1_G^{(n-2)}),g(x_1^{i-1},1_G,x_{i+1}^n),1^{(n-2)})=g(g(x_1^n),r,1^{(n-2)}) \in A$ for each $r \in R$. Since $A$ is an $n$-ary $\delta(0)$-hyperideal and $g(x_1^{i-1},1_G,x_{i+1}^n) \notin \delta(0)$, we obtain $g(x_i,r,1^{(n-2)}) \in A$ which implies $x_i \in E_R$. Thus  $E_R$ is an $n$-ary $\delta(0)$-hyperideal of $G$.
\end{proof}
\begin{theorem}
If $\{A_j\}_{j \in J}$ is a non-empty set of $n$-ary $\delta(0)$-hyperideals of $G$, then so is $\bigcap_{j \in J} A_j.$
\end{theorem}
\begin{proof}
Assume that $g(x_1^n) \in \bigcap_{j \in J}A_j$ for $x_1^n \in G$ such that $g(x_1^{i-1},1_G,x_{i+1}^n) \notin \delta(0)$ for some $i \in \{1,\cdots,n\}$.  Since $A_j$ is an $n$-ary $\delta(0)$-hyperideal of $G$ and $g(x_1^n) \in A_j$ for all $j \in J$, we obtain $x_i \in A_j$ which means $x_i \in \bigcap_{j \in J}A_j$.
\end{proof}
Recall from \cite{mah4} that a proper hyperideal $A$ of $G$ is an $n$-ary $J$-hyperideal if $g(x_1^n) \in A$ for $x_1^n$ and $x_i \notin J(G)$ imply that $g(x_1^{i-1},1_G,x_{i-1}^n) \in A$.
\begin{theorem}
Let $\delta(0)$ be a maximal hyperideal of $G$ and $A$ be an $n$-ary $J$-hyperideal of $G$. Then $A$ is an $n$-ary $\delta(0)$-hyperideal.
\end{theorem}
\begin{proof}
Assume that $g(x_1^n) \in A$ for $x_1^n \in G$ such that $x_i \notin A$ for some $i \in \{1,\cdots,n\}$. Since $g(g(x_1^{i-1},1_G,x_{i+1}^n),x_i,1_G^{(n-2)}) \in A$ and $g(x_i,1_G^{(n-1)}) \notin A$, we conclude that $g(x_1^{i-1},1_G,x_{i+1}^n) \in J(G)$. Since $\delta(0)$ is a maximal hyperideal of $G$, $J(G) \subseteq \delta(0)$ and so $g(x_1^{i-1},1_G,x_{i+1}^n) \in \delta(0)$. Thus $A$ is an $n$-ary $\delta(0)$- hyperideal of $G$.
\end{proof}
\begin{theorem} \label{1}
Let $A$ be an $n$-ary $\delta(0)$-hyperideal of $G$. Then $A$ is contained in $\delta(0)$.
\end{theorem}
\begin{proof}
Suppose that $A$ is an $n$-ary $\delta(0)$-hyperideal of $G$ such that it is not contained in $\delta(0)$. Let $x \in A$ but $x \notin \delta(0)$. Since $A$ is an $n$-ary $\delta(0)$-hyperideal of $G$, $g(x,1_G^{(n-1)}) \in A$ and $g(x,1_G^{(n-1)}) \notin \delta(0)$, we get $1 \in A$ which is a contradiction. Thus $A$ is contained in $\delta(0)$.
\end{proof}
Recall from \cite{sorc1} that a proper hyperideal $A$ of  $G$ is an $n$-ary prime hyperideal if for hyperideals $A_1^n$ of $G$, $g(A_1^ n) \subseteq A$ implies that $A_i \subseteq A$ for some $i \in \{1,\cdots,n\}$. 
In Lemma 4.5 in \cite{sorc1}, it was shown that  a proper hyperideal  $A$ of $G$ is an $n$-ary prime hyperideal if for $x^n_ 1 \in G$, $g(x^n_ 1) \in A$ implies that $x_i \in A$ for some $i \in \{1,\cdots,n\}$.

\begin{remark} \label{1/3}
 $\delta(0)$ is an $n$-ary prime hyperideal of $G$ if and only if $\delta(0)$ is an $n$-ary $\delta(0)$-hyperideal of $G$.
\end{remark}
\begin{theorem} \label{1/2}
Assume that $A$ is an $n$-ary prime hyperideal of $G$ with $\delta(A)=A$. Then $A$ is an $n$-ary $\delta(0)$-hyperideal of $G$  if and only if $A=\delta(0)$.
\end{theorem}
 \begin{proof}
 ($\Longrightarrow$) Since $0 \in A$, we have $\delta(0) \subseteq \delta(A)=A$. Since $A$ is an $n$-ary $\delta(0)$-hyperideal of $G$, the inclusion $A \subseteq \delta(0)$ holds by \ref{1}. Thus $A=\delta(0)$.
 
 ($\Longleftarrow$) Let $A=\delta(0)$ and  $g(x_1^n) \in A$ for $x_1^n \in G$ such that $g(x_1^{i-1},1_G,x_{i+1}^n) \notin \delta(0)$. Since $A$ is an $n$-ary prime hyperideal of $G$ and $g(x_1^{i-1},1_G,x_{i+1}^n) \notin A$, we get $x_i \in A$. Consequently, $A$ is an $n$-ary $\delta(0)$-hyperideal of $G$.
 \end{proof}
 \begin{theorem} \label{1/6} 
Let $A$ be an $n$-ary $\delta(0)$-hyperideal of $G$ such that $\delta(A)=A$. If there is no $\delta(0)$-hyperideal containing $A$ properly, then $A=\delta(0)$.
 \end{theorem}
 \begin{proof}
 It is sufficient to show the hyperideal $A$ is prime. Assume that $g(x_1^n) \in A$ for $x_1^n \in G$ but $x_i \notin A$ for some $i \in \{1,\cdots,n\}$. By Theorem \ref{1/1},  $E_{x_i}=\{x \in G \ \vert \ g(x,x_i,1_G^{(n-2)}) \subseteq A\}$ is an $n$-ary $\delta(0)$-hyperideal of $G$ as $A$ is an $n$-ary $\delta(0)$-hyperideal and $x_i \notin A$. Since there is no $\delta(0)$-hyperideal containing $A$ properly, we get $g(x_1^{i-1},1_G,x_{i+1}^n) \in E_{x_i}=A$. Since $A$ is an $n$-ary $\delta(0)$-hyperideal of $G$, we can continue the process and get $x_j \in A$ for some $j \in \{1,\cdots,n\}$. Then $A$ is an $n$-ary prime hyperideal of $G$. Now, by Theorem \ref{1/2}, we get the result that $A=\delta(0)$.
 \end{proof}
 A commutative Krasner $(m,n)$-hyperring $G$ is said to be an $n$-ary hyperintegral domain if $g(x_1^n)=0$ for $x_1^n \in G$ implies that $x_i=0$ for some $i \in \{1,\cdots,n\}$ \cite{sorc1}. 
 
\begin{remark} \label{1/5}
If $G$ is an $n$-ary  hyperintegral domain, then $0$ is an $n$-ary $\delta(0)$-hyperideal of $G$.
 \end{remark} 
 \begin{proposition} \label{1/4}
If  $G$ is a commutative Krasner $(m,n)$-hypering which is not an $n$-ary hyperintegral domain  and $\delta(0)=0$, then $G$ has no $\delta(0)$-hyperideal.
 \end{proposition}
 \begin{proof}
 Let  $\delta(0)=0$.  Assume that $g(x_1^n) \in \delta(0)$ for $x_1^n \in G$. Since $G$ is not an $n$-ary hyperintegral domain, we get $x_i \neq 0$ for all $i \in \{1,\cdots,n\}$ which means $\delta(0)$ is not an $n$-ary prime hyperideal. Therefore $\delta(0)$ is not an $n$-ary $\delta(0)$-hyperideal of $G$ by Remark \ref{1/3}. Now, assume that $A\neq 0$ is an arbitrary hyperideal of $G$. If $A$ is an $n$-ary $\delta(0)$-hyperideal of $G$, then $A$ is containd in $\delta(0)$ by Theorem \ref{1} which means $A=0$, a contradiction. Consequently, $G$ has no $\delta(0)$-hyperideal.
 \end{proof}
 In view of Remark \ref{1/5} and Proposition \ref{1/4}, we have the following result.
 \begin{corollary}
 Let $\delta(0)=0$. Then $G$ is an $n$-ary hyperintegral domain if and only if $0$ is an $n$-ary $\delta(0)$-hyperideal of $G$.
 \end{corollary}
 Recall from \cite{sorc1} that the radical  of the hyperideal $A$ of $G$, denoted by $rad(A)$, is 
the intersection of all prime hyperideals $G$  containing $A$. If the set of all prime hyperideals containing $A$ is empty, then we define $rad(A)=G$.
 Moreover,  a proper hyperideal $A$ of  $G$  is  an $n$-ary primary hyperideal if $x_1^n \in G$ and $g(x^n _1) \in A$  imply either  $x_i \in A$ or  $g(x_1^{i-1}, 1_G, x_{ i+1}^n) \in rad(A)$ for some $i \in \{1,\cdots,n\}$.
\begin{theorem} \label{2}
Let $A$ be an $n$-ary primary hyperideal of $G$. If $rad(A)$ is containd in $ \delta(0)$, then $A$ is an $n$-ary $\delta(0)$-hyperideal.
\end{theorem}
\begin{proof}
Assume that $A$ is an $n$-ary primary hyperideal of $G$. Let $g(x_1^n) \in A$ for $x_1^n \in G$ but $g(x_1^{i-1},1_G,x_{i+1}^n) \notin \delta(0)$ for some $i \in \{1,\cdots,n\}$. This means $g(x_1^{i-1},1_G,x_{i+1}^n) \notin rad(A)$ and so $x_i \in A$ as $A$ is an $n$-ary primary hyperideal of $G$. Hence $A$ is an $n$-ary $\delta(0)$-hyperideal.
\end{proof}
As an immediate consequence of the previous theorem, we have the following result.
\begin{corollary}
If $A$ is an $n$-ary prime hyperideal of $G$ such that $A$ is contained in $\delta(0)$, then $A$ is an $n$-ary $\delta(0)$-hyperideal.
\end{corollary}
In \cite{sorc1}, it was shown that if $x \in rad(A)$ then 
there exists $r \in \mathbb {N}$ such that $g(x^ {(r)} , 1_G^{(n-r)} ) \in A$ for $r \leq n$, or $g_{(l)} (x^ {(r)} ) \in A$ for $r= l(n-1) + 1$.
\begin{theorem}
Let $0$ be an $n$-ary $\delta(0)$-hyperideal of $G$. If $rad(\delta(0))=\delta(0)$, then $rad(0)$ is an $n$-ary  $\delta(0)$-hyperideal of $G$.
\end{theorem}
\begin{proof}
Let $g(x_1^n) \in rad(0)$ for $x_1^n \in G$ such that $g(x_1^{i-1},1_G,x_{i+1}^n) \notin \delta(0)=rad(\delta(0))$ for some $i \in \{1,\cdots,n\}$. Then we conclude that 
there exists $r \in \mathbb {N}$ such that $g(g(x_1^n)^ {(r)} , 1_G^{(n-r)} )=0$ for $r \leq n$, or $g_{(l)} (g(x_1^n)^ {(r)} )=0$ for $r= l(n-1) + 1$. Let $g(g(x_1^{(r)},1_G^{(n-r)}),\cdots,g(x_n^{(r)},1_G^{(n-r)}))=g(g(x_1^n)^ {(r)} , 1_G^{(n-r)} )=0$ for some $r \leq n$. Since  $g(x_1^{i-1},1_G,x_{i+1}^n) \notin rad(\delta(0))$,  we get the result that 

$g(g(x_1^{(r)},1_G^{(n-r)}),\cdots,g(x_{i-1}^{(r)},1_G^{(n-r)}),1_G,g(x_{i+1}^{(r)},1_G^{(n-r)}),\cdots,g(x_n^{(r)},1_G^{(n-r)})$

$\hspace{1cm}=g(g(x_1^{i-1},1_G,1_{i+1}^n)^{(r)},1_G^{(n-r)})$

$\hspace{1cm} \notin \delta(0)$ 
\\for all $r \leq n$ and $g_{(l)}(g(x_1^{i-1},1_G,x_{i+1}^n)^{(r)}) \notin \delta(0)$ for $r=l(n-1)+1$. Since $0$ is an $n$-ary $\delta(0)$-hyperideal of $G$, we have $g(x_i^{(r)},1_G^{(n-r)})=0$  which means $x_i \in rad(0)$. By using a similar argument, one can easily complete the proof where  $r= l(n-1) + 1$.
\end{proof}
A proper hyperideal $A$ of $G$ is said to be an $n$-ary $\delta$-primary hyperideal if $g(x_1^n) \in A$ for $x_1^n \in G$ implies that $x_i \in A$ or $g(x_1^{i-1},1_G,x_{i-1}^n) \in \delta(A)$ for some $i \in \{1,\cdots,n\}$ \cite{mah3}.
\begin{theorem} Let $A$ be a proper hyperideal of $G$.
\begin{itemize}
\item[\rm{(i)}]~ If $A$ is an $n$-ary $\delta(0)$-hyperideal of $G$, then $A$ is an $n$-ary $\delta$-primary hyperideal.
\item[\rm{(ii)}]~ 
Let $\delta^2(0) \subseteq \delta(0)$. Then $A$ is an $n$-ary $\delta(0)$-hyperideal of $G$ if and only if $A$ is an $n$-ary $\delta$-primary hyperideal and it is contained in $\delta(0)$.
\end{itemize}
\end{theorem}
\begin{proof}
(i) Let  $A$ be an $n$-ary $\delta(0)$-hyperideal of $G$. Assume that $g(x_1^n) \in A$ for $x_1^n \in G$ such that $g(x_1^{i-1},1_G,x_{i+1}^n) \notin \delta(A)$ for some $i \in \{1,\cdots,n\}$. Therefore $g(x_1^{i-1},1_G,x_{i+1}^n) \notin \delta(0)$ as $\delta(0) \subseteq \delta(A)$. Since $A$ is an $n$-ary $\delta(0)$-hyperideal, we get $x_i \in A$. Thus $A$ is an $n$-ary $\delta$-primary hyperideal.

(ii) ($\Longrightarrow$) Let $A$ be an $n$-ary $\delta(0)$-hyperideal of $G$. By (i) and Theorem \ref{1}, we are done.

($\Longleftarrow$) Assume that $A$ is an $n$-ary $\delta$-primary hyperideal and it is contained in $\delta(0)$. Suppose that $g(x_1^n) \in A$ for $x_1^n \in G$ such that $x_i \notin A$ for some $i \in \{1,\cdots,n\}$. Since $A$ is an $n$-ary $\delta$-primary hyperideal and  $A \subseteq \delta(0)$, we conclude that $g(x_1^{i-1},1_G,x_{i+1}^n) \in \delta(A) \subseteq \delta^2(0)$ which means $g(x_1^{i-1},1_G,x_{i+1}^n) \in \delta(0)$. Consequently, $A$ is an $n$-ary $\delta(0)$-hyperideal of $G$.
\end{proof}
\begin{theorem} \label{admits}
If $G$ admits an $n$-ary $\delta(0)$-hyperideal, then there exists an $n$-ary $\delta(0)$-hyperideal of $G$ such that there is no $\delta(0)$-hyperideal containing it. Moreover, if $\delta(A)=A$ for every  $n$-ary $\delta(0)$-hyperideal $A$ that there is no $\delta(0)$-hyperideal containing it, then $\delta(0)$ is an $n$-ary prime hyperideal of $G$.
\end{theorem}
\begin{proof}
Let $A$ be an $n$-ary $\delta(0)$-hyperideal of $G$. Assume that $\Sigma$ is the set of all $n$-ary $\delta(0)$-hyperideals of $G$. From $A \in \Sigma$ it follows that $\Sigma \neq \varnothing$. $\Sigma$ is a parially ordered set with respect to set inclusion relation. Assume that $A_1 \subseteq A_2 \subseteq \cdots$ is some chain in $\Sigma$. Clearly, $\cup_{i=1}^{\infty} A_i$ is a hyperideal of $G$. Assume that $g(x_1^n) \in \cup_{i=1}^{\infty} A_i$ for $x_1^n \in G$ such that $g(x_1^{i-1},1_G,x_{i+1}^n) \notin \delta(0)$ for some $i \in \{1,\cdots,n\}$. This means that there exists $t \in \mathbb{N}$ such that $g(x_1^n) \in A_t$. Since $A_t$ is an $n$-ary $\delta(0)$-hyperideal of $G$ and $g(x_1^{i-1},1_G,x_{i+1}^n) \notin \delta(0)$, we have $x_i \in A_t \subseteq \cup_{i=1}^{\infty}A_i$. Therefore $\cup_{i=1}^{\infty} A_i$ is an upper bound of the mentioned chain. By Zorn$^,$s lemma, there is an $n$-ary $\delta(0)$-hyperideal $\mathbb{A}$ which is maximal in $\Sigma$. Hence we conclude that $\mathbb{A}=\delta(0)$ by Theorem \ref{1/6}.  Thus $\delta(0)$ is  an $n$-ary prime hyperideal of $G$ by Remark \ref{1/3}.
\end{proof}
 
The concept of Krasner $(m,n)$-hyperring of fractions was defined in \cite{mah5}. Let $S$ a non-empty subset  of $G$. $S$  refers to  an $n$-ary multiplicative  subset of $G$, if $1 \in S$ and $g(x_1^n) \in S$ for $x_1,\cdots,x_n \in S$. Suppose that $\delta$ is a hyperideal expansion of a $G$, $S$ is an $n$-ary multiplicative  subset of $G$ and $A$  a hyperideal of $G$. Then $\delta_S$ is a hyperideal expansion of $S^{-1}G$ with $\delta_S(S^{-1}A)=S^{-1}\Big(\delta(A)\Big)$ \cite{mah6}. $S$ is said to be an $n$-ary $\delta(0)$-multiplicative  subset of $G$ if $G-\delta(0) \subseteq S$ and $g(x_1^{n-1},x) \in S$ for all  $x \in S$ and $g(x_1^{n-1},1_G) \in G-\delta(0)$.
\begin{theorem}
Let  $A$ be an $n$-ary $\delta(0_G)$-hyperideal of $(G,f,g)$ and  $S$  an $n$-ary multiplicative  subset of $G$ such that $S\cap A =\varnothing$. Then $S^{-1}A$ is an $n$-ary $\delta_S(0_{S^{-1}G})$-hyperideals of $(S^{-1}G,H,K)$.
\end{theorem}
\begin{proof}
Suppose that $K(\frac{x_1}{s_1},\cdots,\frac{x_n}{s_n}) \in S^{-1}A$ for $\frac{x_1}{s_1},\cdots,\frac{x_n}{s_n} \in S^{-1}G$ such that $K(\frac{x_1}{s_1},\cdots,\frac{x_{i-1}}{s_{i-1}},\frac{1_G}{1_G},\frac{x_{i+1}}{s_{i+1}},\cdots,\frac{x_n}{s_n}) \notin \delta_S(0_{S^{-1}G})$ for some $i \in \{1,\cdots,n\}$. Therefore $\frac{g(x_1^n)}{g(s_1^n)} \in S^{-1}A$ which implies $g(x_1^{i-1},g(s,x_i,1_G^{(n-2)}),x_{i+1}^n)=g(s,g(x_1^n),1_G^{(n-2)}) \in A$ for some $s \in S$. Since $A$ is an $n$-ary $\delta(0_G)$-hyperideal of $G$ and $g(x_1^{i-1},1_G,x_{i+1}^n) \notin \delta(0_G)$, we get the result that $g(s,x_i,1_G^{(n-2)}) \in A$. Hence we have $\frac{x_i}{s_i}=\frac{g(x_i,s,1_G^{(n-2)})}{g(s_i,s,1_G^{(n-2)})}=K(\frac{x_i}{s_i},\frac{s}{s},\frac{1_G}{1_G}^{(n-2)}) \in S^{-1}A$. Thus we conclude that $S^{-1}A$ is an $n$-ary $\delta_S(0_{S^{-1}G})$-hyperideals of $(S^{-1}G,H,K)$.
\end{proof}
\begin{theorem}
Let $A$ be a hyperideal of $G$. Then $A$ is an $n$-ary $\delta(0)$-hyperideal of $G$ if and only if $G-A$ is an $n$-ary $\delta(0)$-multiplicative subset of $G$.
\end{theorem}
\begin{proof}
($\Longrightarrow$)Let $A$ be an $n$-ary $\delta(0)$-hyperideal of $G$. Therefore $A$ is contained in $\delta(0)$ by Theorem \ref{1}. Then we have $G-\delta(0) \subseteq G-A$. Let $x \in G-A$ and $g(x_1^{n-1},1_G) \in G-\delta(0)$. Suppose that $g(x_1^{n-1},x) \in A$. Hence we get $x \in A$ as $A$ is an $n$-ary $\delta(0)$-hyperideal of $G$ and $g(x_1^{n-1},1_G) \notin \delta(0)$. This is a contradiction. Therefore $g(x_1^{n-1},x) \in G-A$ which means $G-A$ is an $n$-ary $\delta(0)$-multiplicative subset of $G$.

($\Longleftarrow$) Let $G-A$ be an $n$-ary $\delta(0)$-multiplicative subset of $G$ for some hyperideal $A$ of $G$. Assume that $g(x_1^n) \in A$ for $x_1^n \in G$ such that $g(x_1^{i-1},1_G,x_{i+1}^n) \notin \delta(0)$ for some $i \in \{1,\cdots,n\}$. If  $x_i \notin A$, then we get $g(x_1^{i-1},x_i,x_{i+1}^n) \notin A$ as $G-A$ is an $n$-ary $\delta(0)$-multiplicative subset of $G$ and $g(x_1^{i-1},1_G,x_{i+1}^n) \in G-\delta(0)$. This is a contradiction. Thus $x_i \in A$. Consequently, $A$ is an $n$-ary $\delta(0)$-hyperideal of $G$.
\end{proof}
\begin{theorem}
Let  $A$ be a hyperideal of $G$ and  $S$  an $n$-ary $\delta(0)$-multiplicative  subset of $G$ such that $S\cap A =\varnothing$. Then there exists an $n$-ary $\delta(0)$-hyperideal $B$ such that $A \subseteq B$ and $B \cap S = \varnothing$.
\end{theorem}
\begin{proof}
Let $\Phi$ be the set of all hyperideals of $G$ that any hyperideal has an empty intersection with $S$. By the hypothesis, $A \in \Phi$ and so $\Phi \neq \varnothing$. $\Phi$ is a partially ordered set with respect to set inclusion relation. Then there exists a maximal element $B$ in $\Phi$ by Zorn$^{,}$s Lemma. We assume that $B$ is not an $n$-ary $\delta(0)$-hyperideal of $G$ and look for a contradiction. The assumption means that 
  $g(x_1^n) \in B$ for $x_1^n \in G$ and neither $x_i \in B$ nor $g(x_1^{i-1},1_G,x_{i+1}^n) \in \delta(0)$ for all $i \in \{1,\cdots,n\}$. Put $E_{g(x_1^{i-1},1_G,x_{i+1}^n)}=\{x \in G \ \vert \ g(x,g(x_1^{i-1},1_G,x_{i-1}^n),1_G^{(n-2)}) \in B\}$ for $i \in \{1,\cdots,n\}$. Since  $B \subsetneq E_{g(x_1^{i-1},1_G,x_{i+1}^n)}$, we conclude that $E_{g(x_1^{i-1},1_G,x_{i+1}^n)} \cap S \neq \varnothing$ by the maximality $B$. Let $t \in E_{g(x_1^{i-1},1_G,x_{i+1}^n)} \cap S$. Then $g(x_1^{i-1},t,x_{i+1}^n)=g(t,g(x_1^{i-1},1_G,x_{i+1}^n),1_G^{(n-2)}) \in B$. On the other hand, since $S$ is an $n$-ary $\delta(0)$-multiplicative subsete of $G$, $t \in S$ and  $g(x_1^{i-1},1_G,x_{i+1}^n) \in G-\delta(0)$, we get the result $g(x_1^{i-1},t,x_{i+1`}^n)=g(t,g(x_1^{i-1},1_G,x_{i+1}^n),1_G^{(n-1)}) \in S$ . This implies that $g(x_1^{i-1},t,x_{i-1}^n) \in B \cap S$ and so $B \cap S \neq \varnothing$. This is a contradiction because $B \in \Phi$. Consequently, $B$ is an $n$-ary $\delta(0)$-hyperideal of $G$. 
\end{proof}
In following theorem, we discuss hyperrings of which every proper hyperideal is an $n$-ary $\delta(0)$-hyperideal.
\begin{theorem}
We have the following equivalent statements:
\begin{itemize} 
\item[\rm{(i)}]~ $G$ is a local Krasner $(m,n)$-hyperring with maximal hyperideal $\delta(0)$.
\item[\rm{(ii)}]~ Every proper principal hyperideal of $G$ is an $n$-ary $\delta(0)$-hyperideal.
\item[\rm{(iii)}]~ Every proper hyperideal of $G$ is an $n$-ary $\delta(0)$-hyperideal.
\end{itemize} 
\end{theorem}
\begin{proof}
(i) $\Longrightarrow$ (ii)  Let $\langle x \rangle$ be a proper  hyperideal of $G$ for $x \in G$. Assume that $g(x_1^n) \in \langle x \rangle$ for $x_1^n \in G$ such that $g(x_1^{i-1},1_G,x_{i+1}^n) \notin \delta(0)$ for some $i \in \{1, \cdots,n\}$. Since $\delta(0)$ is the only maximal hyperideal of $G$ and  $g(x_1^{i-1},1_G,x_{i+1}^n) \notin \delta(0)$, we conclude that  $g(x_1^{i-1},1_G,x_{i+1}^n)$ is invertible. Then there exists $y \in R$ such that $1_G=g(y,g(x_1^{i-1},1_G,x_{i+1}^n),1_G^{(n-2)})$. Hence we have
$x_i=g(x_i,1_G^{(n-1)})=g(x_i, g(y,g(x_1^{i-1},1_G,x_{i+1}^n),1_G^{(n-2)}),1_G^{(n-2)})=g(y,g(x_1^n),1_G^{(n-2)})
 \in \langle x \rangle$. This shows that $\langle x \rangle$ is an $n$-ary $\delta(0)$-hyperideal.

(ii) $\Longrightarrow$ (iii) Assume that $A$ is a hyperideal of $G$. Let $g(x_1^n) \in A$ for $x_1^n \in G$ such that $g(x_1^{i-1},1_G,x_{i+1}^n) \notin \delta(0)$. By (ii), $\langle g(x_1^n) \rangle$ is an $n$-ary $\delta(0)$-hyperideal of $G$. From $g(x_1^n) \in \langle g(x_1^n) \rangle$ it follows that $x_i \in \langle g(x_1^n) \rangle$ as $g(x_1^{i-1},1_G,x_{i+1}^n) \notin \delta(0)$. Since $\langle g(x_1^n) \rangle \subseteq A$, we obtain $x_i \in A$ which implies $A$ is an $n$-ary $\delta(0)$-hyperideal of $G$.

(iii) $\Longrightarrow$ (i) Assume that $A$ is an arbitrary hyperideal of $G$. Hence $A$ is an $n$-ary $\delta(0)$-hyperideal of $G$ by (iii). Therefore we conclude that $A$ is contained in  $ \delta(0)$ by Theorem \ref{1}. Thus $\delta(0)$ is the only maximal hyperideal of $G$ and so $G$ is local.
\end{proof}
Suppose that $(G_1,f_1,g_1)$ and $(G_2,f_2,g_2)$ are two Krasner $(m,n)$-hyperrings. Let $\delta$ and $\gamma$ be hyperideal expansions of $G_1$ and $G_2$, respectively. Then the hyperring homomorphism $\psi: G_1 \longrightarrow G_2$ is said to be a $\delta \gamma$-homomorphism if $\delta(\psi^{-1}(A_2)) = \psi^{-1}(\gamma(A_2))$ for the hyperideal $A_2$ of $G_2$. 
It is remarkable that $\gamma(\psi(A_1)=\psi(\delta(A_1))$ for $\delta \gamma$-epimorphism $\psi$ and for hyperideal $A_1$ of $G_1$ with $ Ker (\psi) \subseteq A_1$  \cite{mah3}. 
Every homomorphism $\psi: G_1 \longrightarrow G _2$ is a $\delta_1 \gamma_1$-homomorphism.
\begin{theorem} \label{1/8} 
Assume that  $G_1$ and $G_2$ are two Krasner $(m,n)$-hyperrings and $\delta$ and $\gamma$ are hyperideal expansions of $G_1$ and $G_1$, respectively,  such that $\psi: G_1 \longrightarrow G_2$ is a $\delta \gamma$-homomorphism.
\begin{itemize} 
\item[\rm{(i)}]~  If $\psi$ is a monomorphism and $A_2$ is an $n$-ary $\gamma(0_{G_2})$-hyperideal of $G_2$, then $\psi^{-1} (A_2)$ is an $n$-ary $\delta(0_{G_1})$-hyperideal of $G_1$.
\item[\rm{(ii)}]~  If $\psi $ is an epimorphism and $A_1$ is an $n$-ary $\delta(0_{G_1})$-hyperideal of $G_1$ such that $Ker(\psi) \subseteq A_1$, then $\psi(A_1)$ is an n-ary $\gamma(0_{G_2})$-hyperideal of $G_2$.
\end{itemize} 
\end{theorem}
\begin{proof}
(i) Let $g_1(x_1^n) \in \psi^{-1} (A_2)$ for $x_1^n \in G_1$ such that $g_1(x_1^{i-1},1_{G_1},x_{i+1}^n) \notin \delta(0_{G_1})$ for some $i \in \{1,\cdots,n\}$. This means that $\psi(g_1(x_1^n))=g_2(\psi(x_1),\cdots,\psi(x_n)) \in A_2$. Since  $\psi$ is a monomorphism and $g_1(x_1^{i-1},1_{G_1},x_{i+1}^n) \notin \delta(0_{G_1})$, we conclude that $g_2(\psi(x_1),\cdots,\psi(x_{i-1}),1_{G_2},\psi(x_{i+1}),\cdots,\psi(x_n))=\psi(g_1(x_1^{i-1},1_{G_1},x_{i+1}^n)) \notin \gamma(0_{G_2})$. Since $A_2$ is an $n$-ary $\gamma(0_{G_2})$-hyperideal of $G_2$, we get $\psi(x_i) \in A_2$ which means $x_i \in \psi^{-1}(A_2)$. Thus $\psi^{-1} (A_2)$ is an $n$-ary $\delta(0_{G_1})$-hyperideal of $G_1$.

(ii) Let  $g_2(y_1^n) \in \psi(A_1)$ for $y_1^n \in G_2$ such that $g_2(y_1^{i-1},1_{G_2},y_{i+1}^n) \notin \gamma(0_{G_2})$ for some $i \in \{1,\cdots,n\}$. Since $\psi$ is an epimorphism, then there exist $x_1^n \in G_1$ such that $\psi(x_1)=y_1,...,\psi(x_n)=y_n$. Thus
$\psi(g_1(x_1^n))=g_2(\psi(x_1),...,\psi(x_n))=g_2(y_1^n) \in \psi(A_1)$.
Since $Ker(\psi) $ is contained in $ A_1$, we have $g_1(x_1^n) \in A_1$.  Since $A_1$ is an $n$-ary $\delta(0_{G_1})$-hyperideal of $G_1$ and $g_1(x_1^{i-1},1_{G_1},x_{i+1}^n) \notin \delta(0_{G_1})$, we get the result that $x_i \in A_1$ which implies $y_i=\psi(x_i) \in \psi(A_1)$. This shows that $\psi(A_1)$ is an n-ary $\gamma(0_{G_2})$-hyperideal of $G_2$.
\end{proof}
Let $A$ be a hyperideal of $G$. The projection map $\pi :G \longrightarrow G/A$, defined by $a \mapsto f(a,A,0^{(m-2)})$ is a homomorphism by Theorem 3.2 in \cite{sorc1}. 
\begin{corollary} 
Let $A$ and $B$ be hyperideals of $G$ such that $A \subseteq B$. If $B$ is an $n$-ary $\delta(0)$-hyperideal of $G$, then $B/A$ is an $n$-ary $\delta_q(0)$-hyperideal of $G/A$. 
\end{corollary}
\begin{proof}
Follows by applying Theorem \ref{1/8} (ii) to the epimorphism $\pi: G \longrightarrow G/A$.
\end{proof}
\begin{corollary} Assume that $B/A$ is an $n$-ary $\delta_q(0)$-hyperideal of $G/A$ where $A$ and $B$ are hyperideals of $G$ with  $A \subseteq B$.
\begin{itemize} 
\item[\rm{(i)}]~ If $A \subseteq \delta(0)$, then $B$ is an $n$-ary $\delta(0)$-hyperideal of $G$.
\item[\rm{(ii)}]~ If $A$ is an $n$-ary $\delta(0)$-hyperideal of $G$, then so is $B$.
\end{itemize} 
\end{corollary}
\begin{proof}
(i) Let $g(x_1^n) \in B$ for $x_1^n \in G$ such that  $g(x_1^{i-1},1_G,x_{i+1}^n) \notin \delta(0)$ for some $i \in \{1,\cdots,n\}$. So  $g(f(x_1,A,0^{(m-2)}),\cdots,f(x_n,A,0^{(m-2)}))=f(g(x_1^n),A,0^{(m-2)}) \in B/A$.  Since  $B/A$ is an $n$-ary $\delta_q(0)$-hyperideal of $G/A$ and 

$g(f(x_1,A,0^{(m-2)}),\cdots,f(x_{i-1},A,0^{(m-2)}),f(1_G,A,0^{(m-2)}),f(x_{i+1},A,0^{(m-2)}),$

$\hspace{2.9cm}\cdots,f(x_n,A,0^{(m-2)}))$

$\hspace{1.5cm}=f(g(x_1^{i-1},1_G,x_{i+1}^n),A,0^{(m-2)})$

$ \hspace{1.5cm} \notin \delta_q(0),$\\
we get the result that $f(x_i,A,0^{(m-2)}) \in B/A$ which means $x_i \in B$. This implies that $B$ is an $n$-ary $\delta(0)$-hyperideal of $G$.

(ii) Since $A$ is an $n$-ary $\delta(0)$-hyperideal of $G$, we conclude that $A$ is contained in $\delta(0)$ by Theorem \ref{1}. Therefore the claim is clear by (i).
\end{proof}
Assume that $(G_1, f_1, g_1)$ and $(G_2, f_2, g_2)$ are two commutative Krasner $(m,n)$-hyperrings such that $1_{G_1}$ and $1_{G_2}$ are scalar identities of $G_1$ and $G_2$, respectively. Then 
 $(G_1 \times G_2, f_1\times f_2 ,g_1 \times g_2 )$ endowed with the following $m$-ary hyperoperation
$f_1\times f_2 $ and $n$-ary operation $g_1 \times g_2$ is a Krasner $(m,n)$-hyperring.
\[f_1 \times f_2\Big((a_1,b_1),\cdots,(a_m,b_m)\Big) = \Big\{(a,b) \ \vert \ \ a \in f_1(a_1^m), b \in f_2(b_1^m)\Big \}\]
$\hspace{1.2cm}
g_1 \times g_2 \Big((x_{1}, y_{1}),\cdots,(x_n,y_n)\Big) =\Big(g_1(x_1^n),g_2(y_1^n)\Big) $,\\
for all $a_1^n,x_1^m \in G_1$ and $b_1^n,y_1^m \in G_2$ \cite{car}. 
\begin{theorem} 
Let $\delta_1$ and $\delta_2$ be hyperideal expansions of the Krasner $(m,n)$-hyperring $G_1$ and $G_2$, respectively,  such that $\delta(0_{G_1 \times G_2})=\delta_1(0_{G_1}) \times \delta_2(0_{G_2})$ and $A_1$  a proper hyperideal of $G_1$. If $A_1 \times G_2$ is an $n$-ary $\delta(0_{G_1 \times G_2})$-hyperideal of $G_1 \times G_2$, then $A_1$ is an $n$-ary $\delta_1(0_{G_1})$-hyperideal of $G_1$.
\end{theorem}
\begin{proof}
Assume that $g_1(x_1^n) \in A_1$ for $x_1^n \in G_1$. This means that 
\[g_1 \times g_2\Big((x_1,1_{G_2}),\cdots,(x_n,1_{G_2})\Big)=\Big(g_1(x_1^n),1_{G_2}\Big) \in A_1 \times G_2.\]
 Since  $A_1 \times G_2$ is an $n$-ary $\delta(0_{G_1 \times G_2})$-hyperideal of $G_1 \times G_2$, we conclude that either 
 $(x_i,1_{G_2}) \in A_1 \times G_2$ or 
 
 $g_1 \times g_2\Big((x_1,1_{G_2}),\cdots,(x_{i-1},1_{G_2}),(1_{G_1},1_{G_2})(x_{i+1},1_{G_2}),\cdots,(x_n,1_{G_2})\Big)$
 
 $\hspace{0.8cm}=\Big( g_1(x_1^{i-1},1_{G_1},x_{i+1}^n),1_{G_2}\Big)$
 
 $\hspace{0.8cm} \in \delta(0_{G_1 \times G_2})$\\
 for some $i \in \{1,\cdots,n\}$. This implies that $x_i \in A_1$ or $g(x_1^{i-1},1_{G_1},x_{i+1}^n) \in \delta_1(0_{G_1})$. Consequently,  $A_1$ is an $n$-ary $\delta_1(0_{G_1})$-hyperideal of $G_1$.
\end{proof}
\section{$(s,n)$-absorbing  $\delta(0)$-hyperideals}
In this section, we present the concept of $(s,n)$-absorbing $\delta(0)$-hyperideals as a generalization of the  $n$-ary $\delta(0)$-hyperideals and a subclass of the $(s,n)$-absorbing $\delta$-primary hyperideals.
\begin{definition} 
Suppose that  $\delta$ is a hyperideal expansion of $G$ and $s \in \mathbb{Z}^+$.  A proper hyperideal $A$ of $G$ is said to be an $(s,n)$-absorbing $\delta(0)$-hyperideal if $g(x_1^{sn-s+1}) \in A$  for $x_1^{sn-s+1} \in G$, then either  $g(x_1^{(s-1)n-s+2}) \in A$ or a $g$-product of $(s-1)n-s+2$ of $x_i^,$s, other than $g(x_1^{(s-1)n-s+2})$, is in $\delta(0)$.
\end{definition}
Obviously, every $n$-ary $\delta(0)$-hyperideal is an $(s,n)$-absorbing $\delta(0)$-hyperideal and every $(s,n)$-absorbing $\delta(0)$-hyperideal is an $(s,n)$-absorbing $\delta$-primary hyperideal. However, these are different notions. 
\begin{example} 
Let us consider the hyperring $(G=[0,a],+,\cdot)$ where $0 < a \leq 1$, $``+"$ is defined by 
\[
p + q=
\begin{cases}
$[0,p]$, & \text{if $p =q$}\\
\{\max\{p,q\}\} & \text{if $p \neq q$}
\end{cases}.\]
and  $``\cdot"$ is the usual multiplication on real numbers. Then $([0,a],f,g)$ is a Krasner $(m,n)$-hyperrin  such that $f(p_1^m)=p_1+\cdots+p_m$ and $g(q_1^n)=q_1\cdots q_n$ for all $p_1^m,q_1^n \in [0,a]$. In the hyperring $0$ is an $(s,n)$-absorbing $\delta_1$-hyperideal. Now, let $a=1$. Then $A=[0,0.5]$ is a $(2,2)$-absorbing $\delta_1$-primary hyperideal. However, it is not $(2,2)$-absorbing $\delta_1(0)$-hyperideals. Take $0.6,0.9 \in G$. We have $g(0.6,0.9,0.9) \in A$, but neither $g(0.9^{(2)},1_G)$ and  $g(0.6,0.9,1,1_G)$ are in $A$ nor  in $\delta_1(0)$.
\end{example}
\begin{theorem}
Let $A$ be an $n$-ary $\delta(0)$-hyperideal of $G$. Then $A$ is a $(2,n)$-absorbing $\delta(0)$-hyperideal of $G$.
\end{theorem}
\begin{proof}
Let $g(x_1^{2n-1}) \in A$ for $x_1^{2n-1} \in G$. Since $A$ is an $n$-ary  $\delta(0)$-hyperideal of $G$, we conclude that $g(x_1^n) \in A$ or $g(x_{n+1}^{2n-1}) \in \delta (0)$. In the second case, we get $g(x_i,x_{n+1}^{2n-1}) \in \delta (0)$, for  $i \in \{1,\cdots,n\}$ because  $\delta(0)$ is a hyperideal of $G$. Thus $A$ is $(2,n)$-absorbing $\delta(0)$-hyperideal of $G$.
\end{proof}
\begin{theorem}
Let $A$ be an $(s,n)$-absorbing $\delta(0)$-hyperideal of $G$. Then $A$ is a $(k,n)$-absorbing $\delta(0)$-hyperideal for $k>n$. 
\end{theorem}
\begin{proof}
Let $g(x_1^{(s+1)n-(s+1)+1}) \in A$ for $x_1^{(s+1)n-(s+1)+1} \in G$. Let us consider  $g(x_1^{n+2})=x$. Since $A$ is $(s,n)$-absorbing $\delta(0)$-hyperideal,  we get the result that $g(x,\cdots,x_{(s+1)n-(s+1)+1}) \in A$ or a $g$-product of $sn-s+1$ of the $x_i^,$s except $g(x,\cdots,x_{(s+1)n-(s+1)+1})$ is in $\delta(0)$. This means that $g(x_i,x_{n+3}^{(s+1)n-(s+1)+1}) \in \delta(0)$ for  $i \in \{1,\cdots n+2\}$ which means $A$ is an $(s+1,n)$-absorbing $\delta(0)$-hyperideal. This shows that  $A$ is a $(k,n)$-absorbing $\delta(0)$-hyperideal for $k>n$.
\end{proof}

\section{weakly $(s,n)$-absorbing $\delta(0)$-hyperideals}  
In this section, we aim to discuss some of the  fundamental results concerning weakly $(s,n)$-absorbing $\delta(0)$-hyperideals as an expansion of the $(s,n)$-absorbing $\delta(0)$-hyperideals.
\begin{definition}
Assume that  $\delta$ is a hyperideal expansion of $G$ and $s \in \mathbb{Z}^+$. A proper hyperideal $A$ of $G$ is called  a weakly $(s,n)$-absorbing $\delta(0)$-hyperideal if $x_1^{sn-t+1} \in G$ and $0 \neq g(x_1^{sn-s+1}) \in A$  imply either  $g(x_1^{(s-1)n-s+2}) \in A$ or a $g$-product of $(s-1)n-s+2$ of $x_i^,$s, other than $g(x_1^{(s-1)n-s+2})$, is in $\delta(0)$.
\end{definition}

\begin{example} 
Cosider Krasner $(m,n)$-hyperring $(G/B,f,g)$ where $G={\mathbb{Z}_3[X,Y,Z]}$,  $B=\langle X^3Y^3Z^3\rangle$ and $f,g$ are ordinary addition and ordinary multiplication, respectively. In the hyperring, $A=rad(B/B)$ is  not a weakly $(1,3)$-absorbing $\delta_1(0_{G/B})$-hyperideal of $G/B$ because $0 \neq 2XYZ+B=(2X+B)(Y+B)(Z+B) \in A$ but none of the elements $(2X+B), (Y+B)$ and $(Z+B)$ are not in $\delta_1(0_{G/B})$.
\end{example}
\begin{proposition}
Assume that $\delta$ and $\gamma$ are hyperideal expansions of $G$ and $A$ is an $(s,n)$-absorbing $\gamma(0)$-hyperideal of $G$. Then  $A$ is a weakly $(s,n)$-absorbing $\gamma o \delta(0)$-hyperideal if and only if  $A$ is an $(s,n)$-absorbing $\gamma o\delta(0)$-hyperideal.
\end{proposition}
\begin{proof}
($\Longrightarrow$) Let $A$ be a weakly $(s,n)$-absorbing $\gamma o \delta(0)$-hyperideal of $G$. Suppose that $g(x_1^{sn-s+1}) \in A$ for $x_1^{sn-s+1} \in G$. If $g(x_1^{sn-s+1}) \neq 0$,  we have $g(x_1^{(s-1)n-s+2}) \in A$ or a $g$-product of $(s-1)n-s+2$ of $x_i^,$s, other than $g(x_1^{(s-1)n-s+2})$, is in $\gamma o \delta(0)$ as $A$ is a weakly $(s,n)$-absorbing $\gamma \circ \delta(0)$-hyperideal of $G$. Let us assume that $g(x_1^{sn-s+1}) = 0 \in A$ such that $g(x_1^{(s-1)n-s+2}) \notin A$. Since $A$ is an $(s,n)$-absorbing $\gamma(0)$-hyperideal of $G$, we get the result that a $g$-product of $(s-1)n-s+2$ of $x_i^,$s, other than $g(x_1^{(s-1)n-s+2})$, is in $\gamma(0)$. Since $\gamma(0) \subseteq \gamma o \delta(0)$,  we conclude that the $g$-product of $(s-1)n-s+2$ of $x_i^,$s is in $\gamma o \delta(0)$. Thus $A$ is an $(s,n)$-absorbing $\gamma o\delta(0)$-hyperideal.

($\Longleftarrow$) It is straightforward.
\end{proof}
We say that a proper hyperideal $A$ of $G$ is called a strongly weakly $(s,n)$-absorbing $\delta(0)$-hyperideal if $0 \neq g(A_1^{sn-s+1}) \subseteq A$ for any  hyperideal $A_1^{sn-n+1}$ of $G$ implies $g(A_1^{(s-1)n-s+2}) \subseteq A$ or a $g$-product of $(s-1)n-s+2$ of $A_i^,$s, other than $g(A_1^{(s-1)n-s+2})$, is a subset of  $\delta(0)$.  Moreover, we say that $(x_1^{s(n-1)+1})$ is  an $(s,n)$-$\delta(0)$-zero of a weakly $(s,n)$-absorbing $\delta(0)$-hyperideal $A$ if $g(x_1^{s(n-1)+1})=0$, $g(x_1^{(s-1)n-s+2}) \notin A$ and all $g$-products of $(s-1)n-s+2$ of $x_i^,$s, except $g(x_1^{(s-1)n-s+2})$  are not in $\delta(0)$. Note that the notation $x_1 ,\cdots, \widehat{x_i},\cdots, x_{sn-s+1}$ indicates that $i$-th term is excluded from the sequence $x_1^{sn-s+1}$.
\begin{theorem} \label{1/9}
Let  $(x_1^{s(n-1)+1})$ be   an $(s,n)$-$\delta(0)$-zero of a strongly weakly $(s,n)$-absorbing $\delta(0)$-hyperideal $A$. Then for $u \in \{1,\cdots,(s-1)n-s+2\}$ and $i_1,\cdots,i_u \in \{1,\cdots,s(n-1)+1\}$, $g(x_1,\cdots,\widehat{x_{i_1}},\cdots,\widehat{x_{i_2}},\cdots,\widehat{x_{i_u}},\cdots,x_{s(n-1)+1},A^{(u)})=0$ 
\end{theorem}
\begin{proof}
We use the induction on $u$. Suppose that  $u=1$. Now we show that  $g(x_1,\cdots,\widehat{x_{i_1}}\cdots,x_{s(n-1)+1},A)=0$.  Let $g(x_1,\cdots,\widehat{x_{i_1}}\cdots,x_{s(n-1)+1},A) \neq 0$. Without loss of generality, we may assume that $g(x_2^{s(n-1)+1},A) \neq 0$. Then we obtain $g(x_2^{s(n-1)+1},a) \neq 0$ for some $a \in A$. Since  $ 0\neq g(x_2^{sn-s+1},f(x_1,a,0^{(m-2)}))=f(g(x_1^{s(n-1)+1}),g(x_2^{s(n-1)+1},a),0^{(m-2)}) \subseteq A$ and $A$ is a strongly weakly $(s,n)$-absorbing $\delta(0)$-hyperideal of $G$, we get the result that $\delta(0)$ contains a $g$-product containing $a$ of the terms $(s-1)n-s+2$ of $x_i^,$s. Then we may assume that $f(g(x_1,x_3^{s(n-1)+1}),g(a,x_3^{s(n-1)+1}),0^{(m-2)})=g(x_3^{s(n-1)+1},f(x_1,a,0^{(m-2)})) \subseteq \delta(0)$. Thus $g(x_1,x_3^{s(n-1)+1})  \in \delta(0)$ as $g(a,x_3^{s(n-1)+1})\in A \subseteq \delta(0)$. This is a contradiction. Therefore, $g(x_1,\cdots,\widehat{x_{i_1}}\cdots,x_{s(n-1)+1},A)= 0$. We suppose that the claim holds for all positive integers which are less than $u>1$ and show that it holds for $u$. Suppose that $g(x_1,\cdots,\widehat{x_{i_1}},\cdots,\widehat{x_{i_2}},\cdots,\widehat{x_{i_u}},\cdots,x_{s(n-1)+1},A^{(u)}) \neq 0$. Let us assume that $g(x_{u+1}^{s(n-1)+1},A^{(u)}) \neq 0$, namely, we omit $x_1^u$. Hence we get $ 0 \neq g(x_{u+1}^{s(n-1)+1},a_1^u) \in A$ for some $x_1^u \in A$. Then $0 \neq g(x_{u+1}^{s(n-1)+1}, f(x_1,a_1,0^{(m-2)}),\cdots,f(x_u,a_u,0^{(m-2)})) \subseteq A$, by induction hypothesis. Thus we obtain

$g(x_{u+1}^{s(n-1)+1},f(x_1,a_1,0^{(m-2)}), \cdots, \widehat{f(x_1,a_1,0^{(m-2)})_{i_1}},\cdots,\widehat{f(x_2,a_2,0^{(m-2)})_{i_2}},$

$\cdots,\widehat{f(x_{n-1},a_{n-1},0^{(m-2)})_{i_{n-1}}},\cdots,\widehat{f(x_u,a_u,0^{(m-2)})}) \subseteq \delta(0)$\\ or

$g(x_{u+1},\cdots,\widehat{x_{i_{u+1}}},\cdots,\widehat{x_{i_{u+2}}},\cdots, \widehat{x_{i_{u+(n-1)}}},\cdots, x_{s(n-1)+1},f(x_1,a_1,0^{(m-2)}),$

$\cdots,f(x_u,a_u,0^{(m-2)}) \subseteq \delta(0)$\\ for some $i \in \{1,\cdots,u\}$. Therefore we get the result that  for some $i \in \{1,\cdots,u\}$, $g(x_{u+1}^{s(n-1)+1},\cdots,x_n^t) \in \delta(0)$ or  $g(x_{u+n}^{s(n-1)+1},\cdots,x_1^u) \in \delta(0)$ or 

$g(x_{u+1},\cdots,\widehat{x_{i_{u+1}}},\cdots,\widehat{x_{i_{u+2}}},\cdots,\widehat{x_{i_{u+s}}},\cdots,x_{s(n-1)+1},f(x_1,a_1,0^{(m-2)}),$

$\cdots, \widehat{f(x_1,a_1,0^{(m-2)})_{i_{u+(s+1)}}}, \widehat{f(x_{n-1-s},a_{n-1-s},0^{(m-2)})_{i_{u+(n-1-s)}}},$

$ \cdots,f(x_u,a_u,0^{(m-2)})) \subseteq \delta(0),$\\ a contradiction because $(x_1^{s(n-1)+1})$ is an $(s,n)$-$\delta(0)$-zero of $A$. Consequently, $g(x_1,\cdots,\widehat{x_{i_1}},\cdots,\widehat{x_{i_2}},\cdots,\widehat{x_{i_u}},\cdots,x_{s(n-1)+1},A^{(u)})=0$.
\end{proof}
Now, we give a copy of Nakayama$^,$s lemma for strongly weakly $(s,n)$-absorbing $\delta(0)$-hyperideals.
\begin{theorem}
Assume that $A$ is a strongly weakly $(s,n)$-absorbing $\delta(0)$-hyperideal but is not $(s,n)$-absorbing $\delta(0)$-hyperideal. Then the following hold:
\begin{itemize}
\item[\rm{(i)}]~ $g(A^{(s(n-1)+1)})=0$
\item[\rm{(ii)}]~ If $M=k(A,1_G^{(n-2)},M)$ for some  $(m,n)$-hypermodule $M$ over $G$ , then $M=0$.
\end{itemize}
\end{theorem}
\begin{proof}
(i) Since $A$ is not $(s,n)$-absorbing $\delta(0)$-hyperideal, we obtain an $(s,n)$-$\delta(0)$-zero $(x_1^{s(n-1)+1})$ of $A$. Let us assume that  $g(A^{(s(n-1)+1)}) \neq 0$. Hence we have $g(a_1^{s(n-1)+1}) \neq 0$ for some $a_1^{s(n-1)+1} \in A$. Therefore we get the result that 

$\hspace{1cm}0 \neq g(f(x_1,a_1,0^{(m-2)}),\cdots,f(x_{s(n-1)+1},a_{s(n-1)+1},0^{(m-2)})) \subseteq A$ \\by Theorem \ref{1/9}. By the hypothesis, we conclude that $g(f(x_1,a_1,0^{(m-2)})^{(s-1)n-s+2}) \subseteq A$ or $\delta(0)$ contains a $g$-product of $(s-1)n-s+2$ of $f(x_i,a_i,0^{(m-2)})^,$s other that $g(f(x_1,a_1,0^{(m-2)})^{(s-1)n-s+2})$. In the first possibilty, we conclude that $A$ contains 
$f(\underbrace{g(x_1^{(s-1)n-s+2}),g(x_1^{u_1},a_1^{u_2}),g(a_1^{(s-1)n-s+2})}_u,0^{(m-u)})$ such that $u_1+u_2=(s-1)n-s+2$. It follows that $g(x_1^{(s-1)n-s+2}) \in A$.  But this is a contradiction
since $(x_1^{s(n-1)+1})$ is an $(s,n)$-$\delta(0)$-zero of $A$. In the second possibility, we may assume that 
$g(f(x_2,a_2,0^{(m-2)}),\cdots,f(x_{(s-1)n-s+3},a_{(s-1)n-s+3},0^{(m-2)}))$ is contained in $\delta(0)$. 
Thus we get the result that 
$f(\underbrace{g(x_2^{(s-1)n-s+3}),g(x_1^{u_1},a_2^{u_2}),g(a_1^{(s-1)n-s+2})}_u,0^{(m-u)})$ is contained in $\delta(0)$ such that $u_1+u_2=(s-1)n-s+2$ which implies $g(x_2^{(s-1)n-s+3}) \in \delta(0)$. It contradicts the fact that  $(x_1^{s(n-1)+1})$ is an $(s,n)$-$\delta(0)$-zero of $A$. Consequently, $g(A^{(s(n-1)+1)})=0$.

(ii) Let $M$ be a $(m,n)$-hypermodule over $G$ such that $M=k(A,1_G^{(n-1)},M)$. By (i), we have $k(g(A^{(s(n-1)+1)}),1_G^{(n-1)},M)=0$. This implies that 

$\hspace{2.5cm}0=k(g(A^{(s(n-1)+1)}),1_G^{(n-1)},M)$

$\hspace{2.8cm}=k(g(A^{(s(n-1))},1_G),A,1_G^{(n-2)},M)$

$\hspace{2.8cm}=k(g(A^{(s(n-1)}),1_G),1_G^{(n-1)},k(A,1_G^{(n-1)},M))$

$\hspace{2.8cm}=k(g(A^{(s(n-1))},1_G),1_G^{(n-1)},M)$

$\hspace{2.8cm}=\cdots$

$\hspace{2.8cm}=k(A,1_G^{(n-1)},k(A,1_G^{(n-1)},M))$

$\hspace{2.8cm}=k(A,1_G^{(n-1)},M)$

$\hspace{2.8cm}=M$
\end{proof}








\end{document}